\documentclass{amsart}
\usepackage[all]{xy}
\usepackage{amsmath,amssymb,dsfont,mathtools,nicefrac}
\usepackage[pdftitle={A functorial equivariant K-theory spectrum and an equivariant Lefschetz formula},
  pdfauthor={Ivo Dell'Ambrogio, Tamaz Kandelaki, Ralf Meyer},
  pdfsubject={Mathematics}
]{hyperref}
\usepackage[lite]{amsrefs}
\usepackage{microtype}

\newcommand*{\MRref}[2]{ \href{http://www.ams.org/mathscinet-getitem?mr=#1}{MR \textbf{#1}}}

\renewcommand{\PrintDOI}[1]{\href{http://dx.doi.org/#1}{DOI #1}%
  \IfEmptyBibField{pages}{, (to appear in print)}{}}

\newtheorem{thm}{Theorem}[section]

\newtheorem{lem}[thm]{Lemma}
\newtheorem{prop}[thm]{Proposition}
\theoremstyle{definition}
\newtheorem{defn}[thm]{Definition}
\theoremstyle{remark}
\newtheorem{rem}[thm]{Remark}
\newtheorem{example}[thm]{Example}
\numberwithin{equation}{section}

\DeclareMathOperator{\Rep}{R}
\DeclareMathOperator{\Hom}{Hom}
\DeclareMathOperator{\tr}{tr}
\DeclareMathOperator{\Ho}{Ho}
\newcommand{\colim}{\mathop{\mathrm{colim}}}

\newcommand*{\hats}{{\widehat{S}}}
\newcommand*{\grot}{\mathbin{\widehat\otimes}}
\newcommand*{\grok}{\widehat\Comp}

\newcommand*{\Unit}{\mathds{1}}
\newcommand*{\Tri}{\mathcal T}
\newcommand*{\Hilb}{\mathcal H}
\newcommand*{\Cat}{\mathcal C}
\newcommand*{\Cliff}{\mathbb C\textup l}
\newcommand*{\KK}{\textup{KK}}
\newcommand*{\KKcat}{\mathfrak{KK}}

\newcommand*{\mini}{\textup{min}}
\newcommand*{\braid}{\textsf{braid}}

\newcommand*{\N}{\mathbb N}
\newcommand*{\Z}{\mathbb Z}
\newcommand*{\R}{\mathbb R}
\newcommand*{\C}{\mathbb C}
\newcommand*{\K}{\textup K}
\newcommand*{\Mat}{\mathbb M}
\newcommand*{\Comp}{\mathcal K}
\newcommand*{\Sphere}{\mathbb S}

\newcommand*{\Sym}{{\mathfrak{Sp}^\Sigma}}
\newcommand*{\KG}{\mathbf K^G}

\newcommand*{\inOb}{\mathrel{\in\in}\nobreak}

\newcommand*{\Ccat}{\mathcal{C}}
\newcommand*{\Dcat}{\mathcal{D}}
\newcommand*{\Cstarcat}{ {\mathfrak{C}^*\!\mathfrak{sep}} }

\newcommand*{\Cst}{\textup{C}^*}
\newcommand*{\Kg}{\textup K^G}

\newcommand*{\Cont}{\textup C}

\newcommand*{\KKt}{\textup{KK}}
\newcommand*{\id}{\textup{id}}
\newcommand*{\Mod}{\mathfrak{Mod}}
\newcommand*{\loc}{\textup{loc}}

\newcommand*{\nb}{\nobreakdash}
\newcommand*{\defeq}{\mathrel{\vcentcolon=}}
\newcommand*{\congto}{\xrightarrow\cong}
\newcommand*{\blank}{\textup{\textvisiblespace}}

\newcommand*{\abs}[1]{\lvert#1\rvert}

\begin{document}
\title[K-theory spectrum and Lefschetz formula]{A functorial equivariant K-theory spectrum and an equivariant Lefschetz formula}

\author{Ivo Dell'Ambrogio}
\address{Universit\"at Bielefeld, Fakult\"at f\"ur Mathematik, BIREP Gruppe, Postfach 10\,01\,31, 33501 Bielefeld, Germany}
\email{ambrogio@math.uni-bielefeld.de}

\author{Heath Emerson}
\address{Department of Mathematics and Statistics, University of Victoria, PO BOX 3045\,STN\,CSC, Victoria, B.\,C., Canada V8W\,3P4}
\email{hemerson@math.uvic.ca}

\author{Tamaz Kandelaki}
\address{A.~Razmadze Mathematical Institute, Tbilisi State University, University Street~2, Tbilisi 380043, Georgia; Tbilisi Centre for Mathematical Sciences, Chavchavadze Ave.~75, 3/35, Tbilisi 0168, Republic of Georgia}
\email{tam.kandel@gmail.com}

\author{Ralf Meyer}
\address{Mathematisches Institut and Courant Centre ``Higher order structures'', Georg-August Universit\"at G\"ottingen, Bunsenstra{\ss}e 3--5, 37073 G\"ottingen, Germany}
\email{rameyer@uni-math.gwdg.de}

\thanks{This research was supported by the Volkswagen Foundation (Georgian--German non-commutative partnership).  Ralf Meyer was supported by the German Research Foundation (Deutsche Forschungsgemeinschaft (DFG)) through the Institutional Strategy of the University of G\"ottingen.}
\subjclass[2000]{19K99, 19K35, 19D55}

\begin{abstract}
  We construct a symmetric spectrum representing the
  \(G\)\nb-equivariant K\nb-theory of \(\Cst\)\nb-algebras for a
  compact group or a proper groupoid~\(G\).  Our spectrum is
  functorial for equivariant \(*\)\nb-homomorphisms.  We use this to
  establish the additivity of the canonical traces for endomorphisms
  of strongly dualisable objects in the bootstrap class in~\(\KKt^G\),
  in analogy to previous results for traces in stable homotopy theory.
  As an application, we prove an equivariant analogue of the Lefschetz
  trace formula for Hodgkin Lie groups.
\end{abstract}

\maketitle

\section{Introduction}
\label{sec:intro}

Let~\(\Ccat\) be a symmetric monoidal category with tensor product~\(\otimes\) and tensor unit~\(\Unit\).  An object~\(A\) of~\(\Ccat\) is called \emph{dualisable} if there is an object~\(A^*\), called its \emph{dual}, and a natural isomorphism
\[
\Ccat(A\otimes B,C) \cong \Ccat(B,A^*\otimes C)
\]
for all objects \(B\) and~\(C\) of~\(\Ccat\).  Such duality isomorphisms exist if and only if there are two morphisms \(\eta\colon \Unit\to A^*\otimes A\) and \(\varepsilon\colon A\otimes A^*\to \Unit\), called unit and counit of the duality, that satisfy two appropriate conditions.  Let \(f\colon A\to A\) be an endomorphism in~\(\Cat\).  Then the \emph{trace} of~\(f\) is the composite endomorphism
\[
\Unit \xrightarrow{\eta} A^*\otimes A \xrightarrow{\braid}
A\otimes A^* \xrightarrow{f\otimes \id_{A^*}}
A\otimes A^* \xrightarrow{\varepsilon} \Unit,
\]
where~\(\braid\) denotes the braiding isomorphism.

We want to compute such traces in the case where~\(\Cat\) is the equivariant Kasparov category \(\KKcat^G\) of separable \(G\)\nb-\(\Cst\)-algebras for a compact group~\(G\).  Here~\(\otimes\) is the minimal \(\Cst\)\nb-algebra tensor product equipped with the diagonal action of~\(G\), and~\(\Unit\) is the \(\Cst\)\nb-algebra of complex numbers; the endomorphism ring of~\(\Unit\) is the representation ring of~\(G\), \(\KK^G_0(\Unit,\Unit) \cong \Rep(G)\).  More generally, our construction still works if~\(G\) is a proper groupoid, in which case \(\otimes\) is the tensor product over the object space~\(G^0\) of the groupoid and \(\Unit=\Cont_0(G^0)\) (see also~\cite{Emerson-Meyer:Dualities}).

If~\(G\) is trivial, \(A=\Cont(X)\) for a smooth compact
manifold and \(f\in \KK(A,A)\) comes from a self-map with
simple isolated fixed points, then the Kasparov product that
gives its trace in \(\KK(\C,\C)=\Z\) may be computed directly
as in~\cite{Emerson-Meyer:Equi_Lefschetz}, and the result is
the Lefschetz number of~\(f\), expressed as a sum of
contributions from the fixed points of~\(f\).  By the Lefschetz
Fixed Point Theorem, this is equal to the alternating sum of
the maps induced by~\(f\) on the rational cohomology groups
of~\(X\).  Here we may replace rational cohomology by
\(\K\)\nb-theory because of the Chern character.  Our goal is
to establish a \(G\)\nb-equivariant generalisation of this
result for suitable compact groups~\(G\).

The trace in the sense of symmetric monoidal categories is the
analogue of the local expression of the Lefschetz invariant in terms
of fixed points.  The analogue of the global homological formula for
the Lefschetz invariant is the graded Hattori--Stallings trace of the
action of~\(f\) on \(\K^G_*(A) \defeq \KK^G_*(\Unit,A)\), viewed as a
module over the ring \(\Rep(G) \defeq \KK^G_*(\Unit,\Unit)\).  This
Hattori--Stallings trace is defined if \(\K^G_*(A)\) has a finite
projective \(\Rep(G)\)-module resolution.

\begin{thm}
  \label{thm:HS_for_KKG}
  Let~\(G\) be a connected compact Lie group with torsion-free fundamental group.  Let~\(A\) be a separable \(G\)\nb-\(\Cst\)-algebra that is non-equivariantly \(\KK\)-equivalent to a commutative \(\Cst\)\nb-algebra.  Assume that \(\K^G_*(A)\) is finitely generated as an \(\Rep(G)\)-module.  Then~\(A\) is dualisable and \(\tr f = \tr \K^G_*(f)\) for any \(f\in \KK^G_0(A,A)\).
\end{thm}

This equivariant Lefschetz Fixed Point Theorem depends on the
following Additivity Theorem for traces:

\begin{thm}
  \label{the:additivity_trace}
  Let \(A\to B\to C\to A[1]\) be an exact triangle in the thick triangulated subcategory of~\(\KKcat^G\) generated by~\(\Unit\) \textup(thus in particular \(A\), \(B\) and \(C\) are dualisable\textup).  If the left square in the following diagram
  \[
  \xymatrix{
    A \ar[r] \ar[d]^{f_A}&
    B \ar[r] \ar[d]^{f_B}&
    C \ar[r] \ar@{.>}[d]^{f_C}&
    A[1] \ar[d]^{f_A[1]}\\
    A \ar[r]&
    B \ar[r]&
    C \ar[r]&
    A[1]
  }
  \]
  commutes, then there is an arrow \(f_C\in \KK_0^G(C,C)\) making the whole diagram commute, and such that \(\tr(f_C) - \tr(f_B) + \tr(f_A) = 0\).
\end{thm}

That traces should be additive in this sense is plausible in any category that is triangulated and symmetric monoidal, provided the tensor product and the triangulated category structure are compatible in a suitable sense.  The compatibility axioms needed for this were first worked out by J.~Peter May in~\cite{May:Additivity}.

Some of the axioms in~\cite{May:Additivity} are, however, amazingly complicated, and already the simplest ones, which require the category in question to be \emph{closed} symmetric monoidal, fail for~\(\KKcat^G\).  This obvious problem may be circumvented by embedding~\(\KKcat^G\) into a larger symmetric monoidal triangulated category that is closed and satisfies May's axioms.  A promising approach to construct such an embedding would use simplicial presheaves on separable \(G\)\nb-\(\Cst\)-algebras, following ideas of Paul Arne {\O}stv{\ae}r~\cite{Ostvaer:Homotopy_Cstar_book}.  Another approach would be to show that \(\KKcat^G\) is part of a derivator and that triangulated categories coming from derivators always satisfy additivity of traces.  But since both approaches require a lot of technical work, we choose a different approach.  It has the disadvantage that it only applies to the subcategory of \(\KKcat^G\) generated by the tensor unit.

Our main tool is to lift \(G\)\nb-equivariant \(\K\)\nb-theory to a symmetric monoidal functor from \(G\)\nb-\(\Cst\)-algebras to a suitable category of module spectra.  This lifting descends to a triangulated functor from \(\KKcat^G\) to the homotopy category of module spectra, which is fully faithful on the thick triangulated subcategory generated by the tensor unit.  Hence we may transport additivity of traces from the category of module spectra, where it is known, to the bootstrap category in \(\KKcat^G\).

There exist several suggestions how to lift \(\K\)\nb-theory to spectra.  First, Ulrich Bunke, Michael Joachim and Stephan Stolz~\cite{Bunke-Joachim-Stolz:Classifying_K} construct an orthogonal \(\K\)\nb-theory spectrum with homotopy groups \(\K_*(A) \cong \KK_*(\C,A)\) using unbounded Kasparov cycles.  However, their construction is only functorial for \emph{essential} \(^*\)\nb-homomorphisms, which means for our purposes that it is \emph{not} functorial.  A later construction by Michael Joachim and Stephan Stolz~\cite{Joachim-Stolz:Enrichment} based on the Cuntz picture is incorrect because the iterated Cuntz algebras are not symmetric with respect to permutations.

Therefore, we provide our own symmetric \(\K\)\nb-theory
spectrum, which is very close to the model
of~\cite{Bunke-Joachim-Stolz:Classifying_K}, but functorial for
arbitrary \(^*\)\nb-homomorphisms.  We use the description of
\(\K\)\nb-theory for graded algebras by Jody
Trout~\cite{Trout:Graded_K}, which may be traced back to the
thesis of Ulrich Haag.

\subsection*{Outline of the article}

In Section~\ref{sec:trace}, we recall some basic notation
regarding Kasparov theory and traces in symmetric monoidal
categories.  In Section~\ref{sec:K-enriched}, we lift
\(G\)\nb-equivariant \(\K\)\nb-theory to a lax monoidal
functor~\(\KG\) from~\(\KKcat^G\) to the category of symmetric
module spectra over the symmetric ring spectrum~\(\KG(\Unit)\).
Our construction works equally well for real \(\Cst\)-algebras,
with no extra costs, so we will cover both cases throughout.
In Section~\ref{sec:additivity}, we prove
Theorem~\ref{the:additivity_trace} on additivity of traces.  In
Section~\ref{sec:application}, we apply additivity of traces to
equivariant Kasparov theory and establish the Lefschetz Fixed
Point Theorem~\ref{thm:HS_for_KKG}.

\section{Traces in symmetric monoidal categories}
\label{sec:trace}

Let \(\Ccat\) be a symmetric monoidal category with tensor product~\(\otimes\), unit object~\(\Unit\), and braiding isomorphisms \(\braid_{A,B}\colon A\otimes B\to B\otimes A\).  We omit from our notation, and usually just ignore, the structural associativity and unit isomorphisms; this is justified by the Coherence Theorem (see \cite{MacLane:CategoriesII}*{Chapter~XI}).

We shall consider the following examples.

\begin{example}
  \label{exa:G-C*sep_smc}
  Let~\(G\) be a locally compact group or, more generally, a locally compact groupoid.
  Let \(\Cstarcat^G\) be the category whose objects are the separable \(G\)\nb-\(\Cst\)-algebras and whose morphisms from~\(A\) to~\(B\) are the \(G\)-equivariant \(^*\)\nb-homomorphisms.
  Given two \(G\)\nb-\(\Cst\)\nb-algebras \(A\) and~\(B\), let \(A\otimes_\mini B\) denote their spatial \(\Cst\)\nb-tensor product.

  If~\(G\) is a group, let \(A\otimes B\) be \(A\otimes_\mini B\) equipped with the diagonal action of~\(G\).
  This defines a symmetric monoidal structure on \(\Cstarcat^G\), where the unit object~\(\Unit\) is~\(\C\) (\(\R\) in the case of real \(\Cst\)-algebras) equipped with the trivial action of~\(G\).
  The braiding isomorphism is the unique extension of the obvious braiding isomorphism on the algebraic tensor product, which is dense in \(A\otimes_\mini B\); the associativity and unit isomorphisms are similarly evident maps.

  If~\(G\) is, more generally, a locally compact groupoid with object space~\(X\), then we modify the above definitions.
  Now \(A\otimes_\mini B\) is canonically a \(\Cst\)\nb-algebra over \(X\times X\).  It yields a \(\Cst\)\nb-algebra over~\(X\) by restricting to the diagonal in~\(X\times X\).
  There is a canonical diagonal action of~\(G\) on this restriction, and the resulting \(G\)\nb-\(\Cst\)-algebra is our tensor product \(A\otimes B\).
  The tensor unit~\(\Unit\) is \(\Cont_0(X)\) equipped with the canonical action of~\(G\).
  (We also use the notation \(A\otimes_X B\) to emphasise the dependence of the tensor product on~\(X\).)
  As before, we get a symmetric monoidal category~\(\Cstarcat^G\).
\end{example}

\begin{example}
  \label{exa:KK_smc}
  Let \(\KKcat^G\) be the category whose objects are the separable \(G\)\nb-\(\Cst\)-algebras and whose morphisms from~\(A\) to~\(B\) are the bivariant \(\K\)\nb-groups \(\KK^G(A,B)\), defined in~\cite{Kasparov:Novikov} for groups and in~\cite{LeGall:KK_groupoid} for groupoids.  There is a canonical functor \(\Cstarcat^G\to \KKcat^G\), which is characterised by a universal property.  Using the latter, one checks easily that the tensor product of Example~\ref{exa:G-C*sep_smc} extends along the canonical functor to a symmetric monoidal structure on~\(\KKcat^G\) with the same object function \((A,B)\mapsto A\otimes B\).  As explained in \cite{Meyer-Nest:BC_Localization}*{Appendix~A}, the category \(\KKcat^G\) is triangulated.  Moreover, \(\Cst\)-direct sums provide countable coproducts in~\(\KKcat^G\), and the tensor bifunctor \(\otimes\) preserves exact triangles and coproducts in each variable.
\end{example}

\begin{example}
  \label{exa:KK_smc_graded}
  Later we will also work with \(\Z/2\)-graded \(\Cst\)\nb-algebras.
  They may be turned into a symmetric monoidal category using the usual tensor product~\(\otimes_\mini\) as above because a \(\Z/2\)-grading is the same thing as a \(\Z/2\)-action.
  But we will instead use the (minimal) \emph{graded} \(\Cst\)\nb-tensor product~\(\grot\).
  In \(A\grot B\), the copies of \(A\) and~\(B\) \emph{graded} commute, that is, the even elements commute with the other tensor factor, and the odd elements in \(A\) and \(B\) anticommute.
  Thus \(A\grot B\) and~\(A\otimes B\) are equal if one of the \(\Cst\)\nb-algebras involved is trivially graded.  More generally, if one of the factors is \emph{evenly graded}, then there is a \(\Cst\)\nb-algebra isomorphism \(A\otimes B \cong A\grot B\).

  The tensor product~\(\grot\) is part of a symmetric monoidal structure.  The tensor unit is \(\C\) (or~\(\R\)) with trivial \(\Z/2\)-grading.  The braiding isomorphism \(\sigma : A\grot B\to B\grot A\) maps \(a\grot b\mapsto (-1)^{\abs{a}\cdot\abs{b}} b\grot a\) for homogeneous \(a\in A\), \(b\in B\).
  This yields symmetric monoidal categories both at the \(^*\)\nb-homomorphism and \(\KK\)-theory level.
  Contrary to the case of trivially graded algebras, \(\KKcat^G\) has no natural triangulation.
\end{example}

\begin{example}
  \label{exa:KK_smc_groupoid_graded}
  Finally, we combine Examples \ref{exa:KK_smc} and~\ref{exa:KK_smc_graded}, considering \(\Z/2\)\nb-graded \(\Cst\)\nb-algebras with an action of a locally compact groupoid~\(G\).  We assume that the grading and \(G\)\nb-action are compatible in the sense that we get an action of \(G\times\Z/2\).  Then the graded tensor product over~\(X\) provides a symmetric monoidal structure on \(\Z/2\)\nb-graded \(G\)\nb-\(\Cst\)\nb-algebras, as well as on the Kasparov category~\(\KKcat^G\).
\end{example}

Let us return to a general symmetric monoidal category~\(\Ccat\).

\begin{defn}
  An object~\(A\) of~\(\Ccat\) is called \emph{dualisable} if there are morphisms
  \[
  \eta\colon \Unit\rightarrow A\otimes A^*\quad\text{and}\quad
  \varepsilon\colon A^*\otimes A\rightarrow \Unit
  \]
  called \emph{unit} (or coevaluation) and \emph{counit} (or evaluation), respectively, which satisfy the \emph{zigzag equations}: each of the composites
  \begin{gather*}
    A \cong \Unit\otimes A
    \xrightarrow{\eta\otimes\id_A} (A\otimes A^*)\otimes A
    \cong A\otimes (A^*\otimes A)
    \xrightarrow{\id_A\otimes \varepsilon} A\otimes\Unit\cong A,\\
    A^* \cong A^*\otimes \Unit
    \xrightarrow{\id_{A^*}\otimes\eta} A^*\otimes (A\otimes A^*)
    \cong (A^*\otimes A)\otimes A^*
    \xrightarrow{\varepsilon\otimes \id_{A^*}} \Unit\otimes A^*
    \cong A^*
  \end{gather*}
  should be the identity.  (The unnamed isomorphisms are the canonical ones.)

  The data \((A,A^*,\eta,\varepsilon)\) is equivalent to natural isomorphisms
  \[
  \Ccat(B\otimes A, C) \cong \Ccat(B, C\otimes A^*)
  \]
  for all objects \(B\) and~\(C\) of~\(\Ccat\).  The dual~\(A^*\) is determined uniquely up to canonical isomorphism.
\end{defn}

In the context of equivariant Kasparov theory, this categorical notion of duality has been considered in~\cite{Skandalis:KK_survey} and, more recently, in~\cite{Echterhoff-Emerson-Kim:Duality}.

\begin{prop}
  \label{prop:dualisable}
  The full subcategory~\(\Ccat_d\) of dualisable objects in~\(\Ccat\)
  is closed symmetric monoidal for the tensor structure of~\(\Ccat\)
  and the internal Hom functor \(A^*\otimes B\).  The duality
  \(A\mapsto A^*\) is a symmetric monoidal equivalence \(\Ccat_d
  \simeq (\Ccat_d)^\textup{op}\).
\end{prop}

For the proof, see \cite{Lewis-etal:Equivariant_stable_homotopy_theory}*{Chapter~III.1}.

\begin{defn}
  \label{def:trace}
  Let~\(A\) be a dualisable object of~\(\Ccat\) and let \(f\colon A\to A\) be an endomorphism in~\(\Ccat\).  The \emph{trace} of~\(f\), denoted \(\tr(f)\colon \Unit\to \Unit\), is the following composition:
  \[
  \xymatrix@R=0pt@C+2em{
    &
    &
    A\otimes A^* \ar[dr]^{\braid} &
    &
    \\
    \Unit \ar[r]^-{\eta} &
    A\otimes A^* \ar[ru]^{f\otimes \id_{A^*}} \ar[rd]_{\braid} &
    &
    A^*\otimes A \ar[r]^-{\varepsilon}  &
    \Unit.
    \\
    &
    &A^*\otimes A \ar[ru]_{\; \id_{A^*}\otimes f} &
    &
  }
  \]
  The \emph{Euler characteristic} of~\(A\) is \(\chi(A) \defeq \tr(\id_A)\).
\end{defn}

\subsection{Lax monoidal functors}

Let \(\Ccat=(\Ccat, \otimes_{\Ccat}, \Unit_{\Ccat})\) and~\(\Dcat =(\Dcat, \otimes_{\Dcat}, \Unit_{\Dcat})\) be two symmetric monoidal categories.  A \emph{\textup(lax\textup) symmetric monoidal functor}~\(F\) from~\(\Ccat\) to~\(\Dcat\) is a functor \(F \colon \Ccat \to\Dcat\) together with a morphism
\[
i\colon \Unit_{\Dcat}\to F(\Unit_{\Ccat})
\]
and a natural transformation
\[
c= c_{A,B}\colon F(A) \otimes_{\Dcat} F(B) \to F(A\otimes_{\Ccat} B)
\quad\quad (A,B\in \Ccat),
\]
which are compatible with the associativity, unit and braiding isomorphisms in \(\Ccat\) and~\(\Dcat\) in a suitable sense (see~\cite{MacLane:CategoriesII}*{Chapter~XI.2} or \cite{Lewis-etal:Equivariant_stable_homotopy_theory}*{p.~126}).
We call~\(F\) (or rather, \((F,c,i)\)) \emph{normal} if~\(i\) is an isomorphism, and \emph{strong} if both \(c\) and~\(i\) are isomorphisms.

\begin{lem}
  \label{lem:strong_monoidal_duals}
  Let \((F,c,i)\colon \Ccat\to \Dcat\) be a normal symmetric monoidal functor, let \(A\in \Ccat\) be dualisable with dual~\(A^*\), and let \(c\colon F(A)\otimes F(A^*)\to F(A\otimes A^*)\) be invertible \textup{(}this happens, in particular, if~\(F\) is strong\textup{)}.
   Then \(F(A)\) is dualisable with dual \(F(A^*)\), and \(\tr F(f) = i^{-1}\circ F(\tr f)\circ i \in \Dcat(\Unit_{\Dcat},\Unit_{\Dcat})\) for all \(f\in \Ccat(A,A)\).
\end{lem}

\begin{proof}
  Let \(\eta\) and~\(\varepsilon\) be the unit and counit for the duality between \(A\) and~\(A^*\).  Then
  \[
  c^{-1}\circ F(\eta)\circ i\colon
  \Unit_\Dcat \to F(A)\otimes F(A^*)
  \quad\text{and}\quad
  i^{-1}\circ F(\varepsilon)\circ c\colon
  F(A^*) \otimes F(A) \to \Unit_\Dcat
  \]
  are unit and counit for \(F(A)\) and~\(F(A^*)\).  The coherence assumptions for symmetric monoidal functors ensure that the zigzag equations carry over and that the traces in \(\Ccat\) and~\(\Dcat\) agree as asserted
   (see \cite{Lewis-etal:Equivariant_stable_homotopy_theory}*{Chapter~III, Proposition~1.9}.)
\end{proof}

\section{Equivariant K-theory as a functor to symmetric spectra}
\label{sec:K-enriched}

Let~\(G\) be a proper, locally compact groupoid with Haar system.  For (\(\Z/2\)-graded) \(G\)\nb-\(\Cst\)\nb-algebras~\(A\), we construct symmetric spectra \(\KG(A)\) with natural isomorphisms
\[
\pi_n\bigl(\KG(A)\bigr) \cong \KK^G_n(\Unit,A)
\qquad\text{for \(n\in \Z\).}
\]
If the orbit space of~\(G\) is compact, then \(\KK^G_n(\Unit,A)\) is
isomorphic to the \(\K\)\nb-theory of the crossed product
\(\Cst\)\nb-algebra \(G\ltimes A\), but not in general.  In
particular, if~\(G\) is a compact group, we have natural isomorphisms
 \[
\pi_n\bigl(\KG(A)\bigr) \cong \KK^G_n(\Unit,A) \cong \K_n(G\ltimes A) \cong \K^G_n(A).
\]

Our construction is close to the one
of~\cite{Bunke-Joachim-Stolz:Classifying_K}, but has the additional
feature that it is functorial for \(G\)-equivariant
\(^*\)\nb-homomorphisms.  Even more, \(\KG(\Unit)\) is a ring
spectrum, \(\KG(A)\) a module spectrum over \(\KG(\Unit)\) for
all~\(A\), and~\(\KG\) comes with the structure of a lax symmetric
monoidal functor between the symmetric monoidal categories of
\(G\)\nb-\(\Cst\)-algebras and \(\KG(\Unit)\)-modules.

\subsection{A suitable description of K-theory}

First we recall the description of \(\K\)\nb-theory in~\cite{Trout:Graded_K}.  This definition uses \(\Z/2\)\nb-graded \(\Cst\)\nb-algebras in a crucial way and, as a benefit, also works for \(\Z/2\)\nb-graded \(\Cst\)\nb-algebras as coefficients.  We note that~\cite{Trout:Graded_K} only considers complex \(\Cst\)\nb-algebras, but all results also hold for real \(\Cst\)\nb-algebras with no additional effort.

Throughout this section, we use the \emph{graded} minimal \(\Cst\)\nb-tensor product~\(\grot\) defined in Example~\ref{exa:KK_smc_graded}.

Let~\(\hats\) be the \(\Cst\)\nb-algebra \(\Cont_0(\R)\) with the \(\Z/2\)\nb-grading by reflection at the origin, so that the even and odd parts of~\(\hats\) consist of the even and odd functions, respectively.

Let~\(\grok\) be the \(\Cst\)\nb-algebra of compact operators on \(\Hilb \defeq \ell^2(\N\times \Z/2)\) with the even grading corresponding to the decomposition \(\Hilb \cong \ell^2(\N)\otimes \delta_0 \oplus \ell^2(\N)\otimes \delta_1\).

For two \(\Z/2\)\nb-graded \(\Cst\)\nb-algebras \(A\) and~\(B\), let \(\Hom(A,B)\) denote the pointed topological space of grading-preserving \(^*\)\nb-homomorphisms from~\(A\) to~\(B\) with the compact-open topology and the zero map as base point.

A pointed continuous map from a pointed compact space~\(X\) to
\(\Hom(A,B)\) is equivalent to an element of \(\Hom(A,\Cont_0(X,B))\),
where \(\Cont_0(X,B)\) denotes the \(\Cst\)\nb-algebra of continuous
functions \(X\to B\) vanishing at the base point of~\(X\)
(see~\cite{Joachim-Johnson:Realizing_KK}*{Proposition 3.4}).  Thus
\(\pi_n(\Hom(A,B))\) is isomorphic to the group of homotopy classes of
grading-preserving \(^*\)\nb-homomorphisms \(A\to \Cont_0(\R^n)\otimes
B\).

It is shown in \cite{Trout:Graded_K}*{Theorem 4.7} that there is a natural bijection
\begin{equation}
  \label{eq:Trout_K_iso}
  \pi_0\bigl(\Hom(\hats, A\grot\grok)\bigr) \cong \KK_0(\C,A) \cong \K_0(A)
\end{equation}
for any graded \(\sigma\)\nb-unital \(\Cst\)\nb-algebra~\(A\).  We need the following generalisation:

\begin{prop}
  \label{pro:K_via_hats}
  Let~\(G\) be a proper locally compact groupoid with Haar system and let~\(A\) be a \(\Z/2\)\nb-graded separable \(G\)\nb-\(\Cst\)\nb-algebra.  Let~\(G^0\) denote the object space of~\(G\) and let \(\Unit\defeq \Cont_0(G^0)\) with the canonical \(G\)\nb-action.  Then there is an isomorphism
  \[
  \pi_n \bigl(\Hom_G(\Unit \grot \hats, A\grot\grok_G)\bigr) \cong
  \KK_n^G(\Unit,A).
  \]
  for every \(n\ge0\), which is natural with respect to grading-preserving \(G\)-equivariant \(^*\)\nb-homomorphisms.
\end{prop}

Here \(\Hom_G(A,B)\subseteq \Hom(A,B)\) denotes the subspace of \(G\)\nb-equivariant maps in \(\Hom(A,B)\) and
\[
\grok_G \defeq \Comp(L^2(G\times \N\times\Z/2)) \cong \Comp(L^2G)\grot\grok.
\]
We have \(\Unit\grot \hats \cong \Unit \otimes \hats \cong \Cont_0(G^0,\hats)\), and there is a \(G\)\nb-equivariant \(\Cst\)\nb-algebra isomorphism \(A\grot \grok_G \cong A\otimes \grok_G\) because~\(\grok_G\) is evenly graded.

\begin{proof}
  The argument in~\cite{Trout:Graded_K} carries over to this situation
  with minor modifications.  If~\(G\) is a compact group,
  then~\(\Unit\) is~\(\C\) or~\(\R\) with trivial \(G\)\nb-action, and
  \begin{multline*}
    \Hom_G(\Unit \grot \hats, A\grot\grok_G)
    = \Hom_G(\hats, A\grot\grok \otimes \Comp(L^2G))
    \\= \Hom\bigl(\hats, (A\grot\grok \otimes \Comp(L^2G))^G\bigr)
    \cong \Hom\bigl(\hats, (A\grot\grok) \rtimes G\bigr),
  \end{multline*}
  where \(B^G\subseteq B\) denotes the \(\Cst\)\nb-subalgebra of \(G\)\nb-invariant elements.  A convenient reference for the isomorphism \((B\otimes\Comp(L^2G))^G \cong B\rtimes G\) is \cite{Guentner-Higson-Trout:Equivariant_E}*{Chapter~11}.  Hence~\eqref{eq:Trout_K_iso} yields
  \begin{multline*}
    \pi_n\bigl(\Hom_G\bigl(\Unit \grot \hats, A\grot\grok_G\bigr)\bigr)
    \cong \pi_0\bigl(\Hom\bigl(\hats, \Cont_0(\R^n) \otimes (A\rtimes G)\grot\grok\bigr)\bigr)
    \\\cong \K_0(\Cont_0(\R^n)\otimes A\rtimes G) \cong \K_n^G(A) \cong \KK_n^G(\C,A).
  \end{multline*}
  Now let~\(G\) be a proper locally compact groupoid with Haar system instead.  The reasoning above carries over almost literally to provide an isomorphism
  \[
  \pi_n\bigl(\Hom_G(\Unit \grot \hats, A\grot\grok_G)\bigr) \cong
  \K_n(A\rtimes G)
  \]
  if~\(G\) is cocompact.  In general, we get the group \(\mathcal{R}\KK_n(G^0/G; \Cont_0(G^0/G), A\rtimes G)\).  In both cases, the result is isomorphic to \(\KK_n^G(\Unit,A)\) by \cite{Tu:Novikov}*{Proposition 6.25}.
\end{proof}

\begin{rem}
  \label{rem:regular}
  We compare our model of Kasparov theory to the one used by Ulrich
  Bunke, Michael Joachim and Stephan
  Scholz~\cite{Bunke-Joachim-Stolz:Classifying_K} for trivial~\(G\).
  They use regular, odd, self-adjoint, unbounded operators~\(D\)
  on \(\Hilb_A\defeq A\otimes \ell^2(\N\times \Z/2)\) that satisfy
  \((1+D^2)^{-1}\in\Comp(\Hilb_A)\).  The functional calculus for such
  an operator is an \emph{essential} grading-preserving
  \(^*\)\nb-homomorphisms from~\(\hats\) to \(A\grot\grok\) and,
  conversely, any such essential homomorphism is of this form for a
  unique~\(D\) as above.  Non-essential grading-preserving
  \(^*\)\nb-homomorphisms from~\(\hats\) to \(A\grot\grok\) correspond
  to unbounded operators on certain hereditary subalgebras of
  \(A\grot\grok\) (see \cite{Trout:Graded_K}*{\S3}).

  Thus the space used in~\cite{Bunke-Joachim-Stolz:Classifying_K} to
  model \(\K\)\nb-theory is homeomorphic to the space of essential,
  grading-preserving \(^*\)\nb-homomorphisms \(\hats\to A\grot\grok\).
  The Kasparov Stabilisation Theorem and standard homotopies of
  isometries show that the restriction to essential homomorphisms does
  not change the homotopy type (see \cites{Meyer:Equivariant,
    Trout:Graded_K}).

  The space of essential morphisms is only functorial for essential
  homomorphisms \(A\to A'\).  The better functoriality of
  \(\Hom(\hats,A\grot\grok)\) is crucial for our purposes.
\end{rem}

Another related reference is~\cite{Haag:Graded}, where the Cuntz picture of Kasparov theory is carried over to the \(\Z/2\)\nb-graded case.  It is shown in~\cite{Haag:Graded}  (for the complex case) that
\[
\KK_0 (A,B) \cong \pi_0 \Hom(\chi A,B\grot\grok)
\]
for a certain \(\Cst\)\nb-algebra~\(\chi A\).
Furthermore, \(\chi \C\) is isomorphic as a graded \(\Cst\)\nb-algebra to \(\widehat{\Mat}_2(\hats)\).  Since the additional stabilisation does not matter, this provides another proof of the isomorphism \(\KK_0(\C,B) \cong \pi_0 \Hom(\hats,B\grot\grok)\).

\subsection{The coalgebra structure\texorpdfstring{ of~\(\hats\)}{}}

It is crucial for our purposes that~\(\hats\) is a counital,
cocommutative, coassociative coalgebra object in the category of
\(\Z/2\)\nb-graded \(\Cst\)\nb-algebras (we work non-equivariantly in
this subsection).  That is, we have a comultiplication \(\Delta\colon
\hats\to\hats\grot\hats\) and a counit \(\epsilon\colon
\hats\to\Unit\) that satisfy the equations
\[
(\Delta\grot\id_\hats)\circ \Delta =
(\id_\hats\grot\Delta)\circ \Delta,\qquad
(\id_\hats\grot\epsilon)\circ \Delta = \id_\hats,\qquad
\braid\circ\Delta = \Delta.
\]
This additional structure is mentioned without proof in~\cite{Higson-Kasparov:Operator_K} in connection with the definition of equivariant E\nb-theory.

In order to understand this structure, we need that essential grading-preserving \(^*\)\nb-homomorphisms \(\hats\to A\) correspond bijectively to odd, self-adjoint, regular unbounded multipliers~\(D\) of~\(A\) with \((1+D^2)^{-1}\in A\) by the functional calculus for regular unbounded self-adjoint multipliers of \(\Cst\)\nb-algebras (\cite{Trout:Graded_K}*{\S3}).

The identity map on~\(\hats\) corresponds to the identical function~\(X\) on~\(\R\), viewed as an unbounded multiplier on~\(\hats\).  In the graded tensor product \(\hats\grot\hats\), the elements \(X\grot1\) and \(1\grot X\) anticommute, so that
\[
(X\grot1+1\grot X)^2 = X^2\grot1+1\grot X^2.
\]
Hence \(X\grot1+1\grot X\) induces an essential grading-preserving \(^*\)\nb-homomorphism
\[
\Delta\colon \hats\to\hats\grot\hats,\qquad
X\mapsto X\grot1+1\grot X.
\]
(The analogous map without gradings is the comultiplication
\[
\Delta\colon \Cont_0(\R) \to \Cont_\textup b(\R\times\R),\qquad
\Delta f(x,y) = f(x+y),
\]
which induces no element in \(\KK_0(\Cont_0\R,\Cont_0\R\otimes\Cont_0\R)\).)

The counit \(\epsilon\colon \hats\to\Unit\) is induced by \(0\in\Unit\).

The equality \((\Delta\grot\id_\hats)\circ\Delta =
(\id_\hats\grot\Delta)\circ\Delta\) holds because both sides are
induced by the unbounded multiplier \(X\grot 1\grot 1 + 1\grot X\grot
1 + 1\grot 1\grot X\).  We have \((\id_\hats\grot\epsilon)\circ\Delta
= \id_\hats\) because \((\id_\hats\grot \epsilon)(X\grot 1+1\grot X) =
X\cdot \epsilon(1) + 1\cdot \epsilon(X) = X\), and
\(\braid\circ\Delta=\Delta\) because \(X\grot 1+1\grot X\) is
symmetric with respect to~\(\braid\).

\subsection{The symmetric module spectra \texorpdfstring{\(\KG(A)\)}{KG(A)}}
\label{subsec:symspcKG}

Our references for symmetric spectra are
\cites{Schwede:Untitled_symmetric, Schwede:Homotopy_symmetric}.  We
will work with the category~\(\Sym\) of symmetric spectra based on
compactly generated spaces; accordingly, all relevant spaces and
constructions --~such as limits and colimits~-- take place in the
category of (pointed) compactly generated spaces (see
\cite{Hovey:Model_cats}*{\S2.4}).  The convenient definitions in
\cite{Schwede:Untitled_symmetric}*{Definitions I.1.1--4} involve a
minimal collection of data for symmetric spectra, symmetric ring
spectra, and modules over a symmetric spectrum.

We are going to construct a symmetric spectrum~\(\KG(A)\) for any
\(\Z/2\)\nb-graded \(G\)\nb-\(\Cst\)-algebra~\(A\).  Even more,
\(\KG(\Unit)\) will be a commutative symmetric ring spectrum
and~\(\KG(A)\) a symmetric module spectrum over~\(\KG(\Unit)\), and
altogether our symmetric spectra will form a symmetric lax monoidal
functor from the category of graded \(G\)\nb-\(\Cst\)-algebras to the
category of \(\KG(\Unit)\)\nb-modules.  To get all this
structure, we need:
\begin{itemize}
\item pointed spaces \(\KG_n(A)\) with pointed continuous \(\Sigma_n\)\nb-actions, where~\(\Sigma_n\) denotes the symmetric group, for \(n\ge0\) and separable \(G\)\nb-\(\Cst\)\nb-algebras~\(A\);
\item \(\Sigma_n\)\nb-equivariant, pointed, continuous maps \(\KG_n(f)\colon \KG_n(A)\to\KG_n(B)\) for \(G\)\nb-equivariant \(^*\)\nb-homomorphisms \(A\to B\);
\item \(\Sigma_n\times\Sigma_m\)-equivariant pointed maps
  \[
  c_{n,m}\colon \KG_n(A_1) \wedge \KG_m(A_2) \to \KG_{n+m}(A_1\grot A_2)
  \]
  for \(n,m\ge0\) and separable \(G\)\nb-\(\Cst\)\nb-algebras \(A_1\) and~\(A_2\);
\item unit maps \(\iota_0\colon \Sphere^0\to\KG_0(\Unit)\) (that is,
  \(\iota_0\in \KG_0(\Unit)\)) and \(\iota_1\colon
  \Sphere^1\to\KG_1(\Unit)\), where~\(\Sphere^n\) denotes the pointed
  \(n\)\nb-sphere.
\end{itemize}
First, we construct this data and check the relevant properties only in the case where~\(G\) is a compact \emph{group}.  The groupoid case is essentially the same, but notationally more complicated.

We will identify \(\R^m \cong \Sphere^m\setminus\{*\}\) where~\(*\) is
the base point.  Our construction uses the Clifford
algebra~\(\Cliff_\R\) of~\(\R\), which is the unital \(\Z/2\)\nb-graded
\(\Cst\)\nb-algebra with basis \(1,F\) where~\(F\) is an odd,
self-adjoint, involution (\(F^2=1\)).  This \(\Cst\)\nb-algebra plays
the role of a formal desuspension for (real or complex)
\(\K\)\nb-theory: there is an invertible element in
\(\KK_0\bigl(\Unit,\Cont_0(\R,\Cliff_\R)\bigr)\) (see~\cite{Kasparov:Operator_K}).

For a separable \(\Z/2\)\nb-graded \(G\)\nb-\(\Cst\)\nb-algebra~\(A\) and \(n\ge0\), let
\begin{equation}
  \label{equa:symmspace}
  \KG_n(A) \defeq \Hom_G\bigl(\hats, A\grot (\Cliff_\R\grot\grok_G)^{\grot n}\bigr),
\end{equation}
where~\(B^{\grot n}\) denotes the \(\Z/2\)\nb-graded tensor product of
\(n\)~copies of~\(B\).  This is a pointed, compactly generated
topological space.  Let~\(\Sigma_n\) act trivially on~\(\hats\)
and~\(A\), and on \((\Cliff_\R\grot\grok_G)^{\grot n}\) by the
permutation action from the braiding of~\(\grot\) (this involves signs
according to the Koszul sign rule, see
Example~\ref{exa:KK_smc_graded}).

\begin{rem}
  If~\(A\) is evenly graded, then \(\Cliff_\R\grot A\) is isomorphic as a \(\Z/2\)\nb-graded \(\Cst\)\nb-algebra to \(\Cliff_\R\otimes A\) with the grading coming only from~\(\Cliff_\R\).  Hence we may replace \(\Cliff_\R\grot\grok_G\) above by \(\Cliff_\R\otimes\Comp(\ell^2(\N\times G))\).
\end{rem}

A grading-preserving \(G\)\nb-equivariant \(^*\)\nb-homomorphism \(f\colon A_1\to A_2\) induces pointed, continuous \(\Sigma_n\)\nb-equivariant maps
\begin{equation}
  \label{eq:KG_functorial}
  \KG_n(f) \colon \KG_n(A_1) \to \KG_n(A_2),\qquad
  \alpha\mapsto (f\grot\id)\circ \alpha.
\end{equation}

Let \(A_1\) and~\(A_2\) be \(\Z/2\)\nb-graded separable \(G\)\nb-\(\Cst\)-algebras and let \(\alpha_1\in\KG_n(A_1)\), \(\alpha_2\in\KG_m(A_2)\).  We get a grading-preserving \(G\)\nb-equivariant \(^*\)\nb-homomorphism
\[
\alpha_1\grot\alpha_2\colon \hats\grot\hats \to A_1\grot(\Cliff_\R\grot\grok_G)^{\grot n} \grot A_2\grot (\Cliff_\R\grot\grok_G)^{\grot m}.
\]
Let~\(\pi_{n,m}\) be the braiding isomorphism that reorders the
tensor factors in the target \(\Cst\)\nb-algebra to \(A_1\grot
A_2\grot (\Cliff_\R\grot\grok_G)^{\grot n+m}\), without
changing the order among the factors \(\Cliff_\R\grot\grok_G\).
We define a \(\Sigma_n\times\Sigma_m\)-equivariant pointed
continuous map
\begin{multline}
  \label{eq:KG_product}
  c_{n,m}^{A_1,A_2}\colon \KG_n(A_1) \wedge \KG_m(A_2) \to
  \KG_{n+m}(A_1\grot A_2),
  \\ \alpha_1\wedge\alpha_2\mapsto \pi_{n,m}\circ (\alpha_1\grot\alpha_2)\circ\Delta\colon
\hats \to A_1\grot A_2\grot (\Cliff_\R\grot\grok_G)^{\grot n+m}.
\end{multline}
In particular, for \(A_1=A_2=\Unit\), this yields the maps
\[
c_{n,m}\colon \KG_n(\Unit) \wedge \KG_m(\Unit) \to \KG_{n+m}(\Unit\grot\Unit) \cong \KG_{n+m}(\Unit)
\]
that are needed for the structure of a symmetric ring spectrum, and
for \(A_1=A\) and \(A_2=\Unit\), this yields the maps
\[
c_{n,m}\colon \KG_n(A) \wedge \KG_m(\Unit) \to \KG_{n+m}(A\grot\Unit) \cong \KG_{n+m}(A)
\]
that are needed for the structure of a \(\KG(\Unit)\)-module.

The unit \(\iota_0\in \KG_0(\Unit) \cong \Hom(\hats,\Unit)\) is the counit~\(\epsilon\) of~\(\hats\).

The other unit~\(\iota_1\) requires two ingredients.  The function
\(t\mapsto t\cdot F\) defines an unbounded, regular, self-adjoint, odd
multiplier~\(D\) of \(\Cont_0(\R,\Cliff_\R)\) with \((1+D^2)^{-1}\colon
t\mapsto (1+ t^2)^{-1}\in \Cont_0(\R,\Cliff_\R)\).  Hence the
functional calculus for~\(D\) is a \(^*\)\nb-homomorphism
\[
\beta\colon \hats \to \Cont_0(\R,\Cliff_\R).
\]
This describes an invertible element in
\(\KK(\Unit,\Cont_0(\R,\Cliff_\R))\), see~\cite{Kasparov:Operator_K}.

Moreover, let \(\gamma\colon \Unit\to\grok_G\) be the
\(^*\)\nb-homomorphism that maps the unit element in \(\R\) or~\(\C\)
to the rank-one-projection onto the subspace spanned by the
\(G\)\nb-invariant vector \(1_G\otimes\delta_0\).  Let
\[
\iota_1\defeq \beta\grot\gamma\colon \hats \cong \hats\grot\Unit \to
\Cont_0(\R,\Cliff_\R)\grot\grok_G.
\]
This is an element of
\[
\Hom_G\bigl(\hats,\Cont_0(\Sphere^1,\Cliff_\R\grot\grok_G)\bigr)
\cong \Hom\bigl(\Sphere^1,\KG_1(\Unit)\bigr)
\]

Now we generalise to the case where~\(G\) is a proper locally compact groupoid with Haar system.  Let~\(X\) denote its object space.  We replace~\(\hats\) by
\[
\hats_X \defeq \hats\grot\Cont_0(X) \cong \Cont_0(X,\hats).
\]
Since the tensor product in the category of graded \(G\)\nb-\(\Cst\)\nb-algebras is taken over~\(X\), we have canonical isomorphisms
\[
\hats_X^{\grot_X n} \cong \Cont_0(X,\hats^{\grot n})\qquad\text{for all \(n\ge0\).}
\]
Thus the coalgebra structure on~\(\hats\) turns~\(\hats_X\) into a coassociative, cocommutative, counital coalgebra object in the symmetric monoidal category of \(\Z/2\)-graded \(G\)\nb-\(\Cst\)-algebras (as in Example~\ref{exa:KK_smc_groupoid_graded}).

Now we may extend the definitions above almost literally, replacing~\(\hats\) by~\(\hats_X\) and~\(\grot\) by~\(\grot_X\) where necessary.  That is, we let
\[
\KG_n(A) \defeq \Hom_G(\hats_X, A\grot_X (\Cliff_\R\grot\grok_G)^{\grot_X n}).
\]
The maps \(\KG_n(f)\) for a grading-preserving \(G\)\nb-equivariant \(^*\)\nb-homomorphism~\(f\) and the maps \(c_{n,m}^{A_1,A_2}\) are defined as above, and \(\iota_0\in\KG_0(\Unit)\) is the counit of~\(\hats_X\).

In the groupoid case, \(L^2G\) is a continuous field of Hilbert spaces over~\(X\), that is, a Hilbert module over \(\Cont_0(X) = \Unit\).  Since~\(G\) is proper, the fibrewise constant functions belong to~\(L^2G\).  They form a one-dimensional subfield of~\(L^2G\) and yield an embedding \(\Cont_0(X)\to \Comp(L^2G)\).  Tensoring with the embedding from a rank-one-projection in~\(\grok\), we get an embedding \(\gamma\colon \Cont_0(X)\to \Comp(L^2G)\grot\grok=\grok_G\).  As above, this yields
\[
\iota_1\defeq \beta\grot\gamma\colon \hats_X \to
\Cont_0(\R,\Cliff_\R)\grot\grok_G,
\]
which we view as a continuous map \(\iota_1\colon \Sphere^1\to\KG_1(\Unit)\).

\begin{lem}
  \label{lem:properties_c}
  The maps~\(c_{n,m}^{A_1,A_2}\) are natural, associative, commutative, and unital with respect to~\(\iota_0\).
\end{lem}

\begin{proof}
  Naturality is obvious.  It means that
  \[
  c_{n,m}^{B_1,B_2}\circ (\KG_n(f_1)\wedge \KG_m(f_2))
  = \KG_{n+m}(f_1\grot_X f_2) \circ c_{n,m}^{A_1,A_2}
  \]
  for grading-preserving \(G\)\nb-equivariant \(^*\)\nb-homomorphisms \(f_j\colon A_j\to B_j\), \(j=1,2\).

  Associativity means that the following square commutes:
  \[
  \xymatrix@C+3em{
    \KG_n(A_1)\wedge \KG_m(A_2)\wedge \KG_p(A_3) \ar[r]^-{c^{A_1,A_2}_{n,m}\wedge\id} \ar[d]_{\id\wedge c^{A_2,A_3}_{m,p}}&
    \KG_{n+m}(A_1 \grot_X A_2)\wedge \KG_p(A_3) \ar[d]^{c^{A_1\grot_X A_2,A_3}_{n+m,p}}\\
    \KG_n(A_1)\wedge \KG_{m+p}(A_2 \grot_X A_3) \ar[r]^-{c^{A_1, A_2 \grot_X A_3}_{n,m+p}}&
    \KG_{n+m+p}(A_1 \grot_X A_2\grot_X A_3)
  }
  \]
  Indeed, both compositions map \(\alpha_1\wedge\alpha_2\wedge\alpha_3\) to the composite map
  \begin{multline*}
    \hats_X \xrightarrow{\Delta^2} \hats_X\grot_X\hats_X\grot_X\hats_X
    \xrightarrow{\alpha_1\grot_X\alpha_2\grot_X\alpha_3}\\
    A_1 \grot_X (\Cliff_\R\grot\grok_G)^{\grot_X n}\grot
    A_2 \grot_X (\Cliff_\R\grot\grok_G)^{\grot_X m}\grot
    A_3 \grot_X (\Cliff_\R\grot\grok_G)^{\grot_X p}
    \\\xrightarrow[\cong]{\pi}
    A_1 \grot_X A_2 \grot_X A_3 \grot_X (\Cliff_\R\grot\grok_G)^{\grot_X n+m+p},
  \end{multline*}
  where the last map is the braiding isomorphism for the permutation that does not change the order among the factors \(\Cliff_\R\grot\grok_G\).  This argument uses the coassociativity of~\(\Delta\) and that~\(\grot_X\) is a symmetric monoidal structure.

  Commutativity means that the following diagram commutes:
  \[
  \xymatrix@C+0em{
    \KG_n(A_1) \wedge \KG_m(A_2)\ar[d]^{c_{n,m}}\ar[rrr]^-{\braid}&&&
    \KG_m(A_2) \wedge \KG_n(A_1) \ar[d]_{c_{m,n}}\\
    \KG_{n+m}(A_1\grot_X A_2)\ar[rr]^-{\KG_{n+m}(\braid)}&&\KG_{n+m}(A_2\grot_X A_1)
    \ar[r]^-{\chi_{n,m}}&\KG_{m+n}(A_2\grot_X A_1),
  }
  \]
  where~\(\chi_{n,m}\) denotes the action of the shuffle permutation that moves the first~\(n\) entries to the end without changing the order among the first~\(n\) and the last~\(m\) entries.  This follows from the cocommutativity of~\(\Delta\); the permutation~\(\chi_{n,m}\) appears because replacing \(\alpha_1\grot_X\alpha_2\) by \(\alpha_2\grot_X\alpha_1\) for \(\alpha_j\in\KG_n(A_j)\), \(j=1,2\), exchanges both the factors \(A_1\) and~\(A_2\) and \((\Cliff_\R\grot\grok_G)^{\grot_X n}\) and \((\Cliff_\R\grot\grok_G)^{\grot_X m}\).

  The right unit property for~\(\iota_0\) means that the composition
  \[
  \KG_n(A)\wedge\Sphere^0 \xrightarrow{\id\wedge\iota_0}
  \KG_n(A)\wedge \KG_0(\Unit) \xrightarrow{c}
  \KG_n(A\grot \Unit) \cong
  \KG_n(A)
  \]
  is the canonical homeomorphism \(\KG_n(A)\wedge\Sphere^0 \cong
  \KG_n(A)\).  This follows because~\(\varepsilon\) is a counit
  for~\(\Delta\).  The left unit property follows from this and
  commutativity.
\end{proof}

Lemma~\ref{lem:properties_c} implies that the \(\Sigma_n\)\nb-spaces \(\KG_n(\Unit)\) with the multiplication maps~\(c_{n,m}^{\Unit,\Unit}\) and the units \(\iota_0\) and~\(\iota_1\) form a commutative, symmetric ring spectrum \(\KG(\Unit)\) in the sense of \cite{Schwede:Untitled_symmetric}*{Definitions I.1.3} (the centrality condition follows from commutativity), and that the \(\Sigma_n\)\nb-spaces \(\KG_n(A)\) with the action maps~\(c_{n,m}^{A,\Unit}\) form a right \(\KG(\Unit)\)-module \(\KG(A)\) in the sense of \cite{Schwede:Untitled_symmetric}*{Definitions I.1.4}.  Furthermore, the maps \(\KG_n(f)\) induced by a grading-preserving \(^*\)\nb-homomorphism \(f\colon A\to B\) form a morphism of \(\KG(\Unit)\)-modules \(\KG(f)\colon \KG(A)\to\KG(B)\).

The right \(\KG(\Unit)\)-modules form a symmetric monoidal category with respect to the smash product~\(\wedge_{\KG(\Unit)}\), which is defined so that a map \(X\wedge_{\KG(\Unit)} Y\to Z\) for \(\KG(\Unit)\)-modules \(X\), \(Y\) and~\(Z\) is a family of maps \(X_n\wedge Y_m\to Z_{n+m}\) that is \(\KG(\Unit)\)\nb-bilinear in the sense that the following two diagrams commute:
\[
\xymatrix{
   X_n\wedge Y_m \wedge \KG(\Unit)_p \ar[r]\ar[d]&
  X_{n}\wedge Y_{m+p}\ar[d]\\
   Z_{n+m}\wedge\KG(\Unit)_p \ar[r]&
  Z_{n+m+p}
}
\]
\[
\xymatrix{
  X_n\wedge Y_m\wedge \KG(\Unit)_p \ar[r]^{\braid}\ar[d]&
  X_n\wedge \KG(\Unit)_p\wedge Y_m \ar[r]&
  X_{n+p}\wedge Y_{m}\ar[d]\\
  Z_{n+m} \wedge  \KG(\Unit)_p\ar[r]&
  Z_{n+m+p} \ar[r]^{\chi_{m,p}}_{\cong}&
  Z_{n+p+m}
}
\]
where~\(\chi_{m,p}\) is the shuffle permutation of the last \(m\)
and~\(p\) numbers.  The tensor unit for~\(\wedge_{\KG(\Unit)}\) is
\(\KG(\Unit)\).

Lemma~\ref{lem:properties_c} implies that the maps~\(c_{n,m}^{A_1,A_2}\) are \(\KG(\Unit)\)-bilinear in this sense.  Thus they produce \(\KG(\Unit)\)-module homomorphisms
\[
c^{A_1,A_2}\colon \KG(A_1) \wedge_{\KG(\Unit)} \KG(A_2) \to \KG(A_1\wedge A_2).
\]

\begin{prop}
  The functor~\(\KG\) with the maps~\(c^{A_1,A_2}\) and the identity map on \(\KG(\Unit)\) is a normal, symmetric lax monoidal functor from the category of graded \(G\)\nb-\(\Cst\)-algebras to the category of \(\KG(\Unit)\)-modules.
\end{prop}

\begin{proof}
  The compatibility conditions we need are contained in Lemma~\ref{lem:properties_c}.
\end{proof}

Proposition~\ref{pro:K_via_hats} yields
\begin{equation}
  \label{htpy_n}
  \pi_{k+n}  \KG_n(A)
  \cong \KK_{k+n}^G (\Unit, (\Cliff_\R\grot \grok_G)^n)
  \cong \KK_k^G(\Unit,A)
\end{equation}
for all \(n\ge1\) and \(k\ge -n\).  Furthermore, since~\(\iota_1\) is constructed from the generator of \(\KK_0(\Unit,\Cont_0(\R, \Cliff_\R))\), the map \(\KG_n(A)\to\Omega\KG_{n+1}(A)\) that it induces is a weak equivalence for all \(n\ge1\).  Thus \(\KG(A)\) is a positive \(\Omega\)\nb-spectrum in the sense of \cite{Schwede:Untitled_symmetric}*{Definition I.1.9}.  This implies that \(\KG(A)\) is semistable, so that~\eqref{htpy_n} computes the correct homotopy groups of the spectrum~\(\KG(A)\), namely, the Hom groups out of the sphere spectra~\(\Sphere[k]\) (see \cite{Schwede:Homotopy_symmetric}*{Example 4.2}, \cite{Hovey-Shipley-Smith:SymmSp}*{Proposition 5.6.4\,(2)}).  
Thus for all \(k\in \Z\) we get
\begin{multline}
  \label{eq:KG_homotopy}
  \Ho \Sym(\Sphere[k], \KG(A))
  \cong \pi_k \KG(A)
  = \colim_{n\to\infty} \pi_{k+n} \KG_n(A)
  \\\cong \colim_{n\to\infty}  \KK_{(k+n) - n}^G(\Unit, A)
  = \KK_k^G(\Unit, A).
\end{multline}

\begin{rem}
\label{rem:extra}
The functor~\(\KG\) is also compatible with other extra structure:
\begin{enumerate}
\item It commutes with pull-backs, that is, it maps the pull-back of a diagram \(A_1\to B\leftarrow A_2\) to the pull-back of \(\KG(A_1)\to \KG(B)\leftarrow \KG(A_2)\).
\item It maps suspensions to loop spaces, that is,
  \[
  \KG(\Cont_0(\R)\otimes A) \cong
  \KG(\Cont_0(\R, A)) \cong \Omega\KG(A)
  \]
  because a \(^*\)\nb-homomorphism to \(A\to\Cont_0(\R,B)\) is a pointed continuous map \(\Sphere^1\to\Hom(A,B)\).
\item It maps the cone of a grading-preserving \(G\)\nb-equivariant \(^*\)\nb-homomorphism \(f\colon A\to B\) to the homotopy fibre of \(\KG(f)\).  Recall that the cone of~\(f\) is
  \[
  C_f \defeq \{(a,b)\in A\oplus \Cont_0((-\infty,\infty],B) : f(a) = b(\infty)\}.
  \]
  An element of \(\KG_n(C_f)\) is a pair \((a,b)\), where \(a\in\KG_n(A)\) and~\(b\) is a path in \(\KG_n(B)\) starting at the basepoint and ending at \(\KG_n(f)(a)\).  This is the homotopy fibre of~\(\KG(f)\), compare \cite{Mandell-etal:Diagram_spectra}*{Definition~6.8}.
\end{enumerate}
\end{rem}

\subsection{Extension to KK-theory}
\label{sec:KG_on_KKG}

So far, the functor~\(\KG\) is defined on the level of \(\Cst\)\nb-algebras and \(^*\)\nb-homomorphisms, but we need a functor defined on Kasparov theory.  For this, we use the universal property of Kasparov theory.  This requires that we restrict attention to trivially graded \(\Cst\)\nb-algebras.  That is, we now consider the symmetric monoidal triangulated category \(\KKcat^G\) as in Example~\ref{exa:KK_smc}.

The universal property of Kasparov theory is formulated in the non-equivariant case in~\cite{Higson:Characterization_KK}.  It is extended to the equivariant case for locally compact groups by Thomsen.  The proof of the universal property in~\cite{Meyer:Equivariant} for locally compact groups also works for locally compact Hausdorff groupoids with Haar system.  We notice the following consequence:

\begin{thm}
  \label{thm:factor_through_KK}
  Let~\(G\) be a proper, locally compact, Hausdorff groupoid with Haar system.  If a functor \(\Cstarcat^G\to\Cat\) maps all \(\KK^G\)\nb-equivalences to isomorphisms, then it factors through the canonical functor \(\Cstarcat^G \to \KKcat^G\).
\end{thm}

\begin{proof}
  The set \(\KK^G(A,B)\) is naturally isomorphic to the set of homotopy classes of \(G\)\nb-equivariant \(^*\)\nb-homomorphisms \(qA\to \Comp(L^2G\otimes \ell^2\N)\otimes B\) (we do not have to stabilise~\(A\) because~\(G\) is proper, compare \cite{Meyer:Equivariant}*{Theorem~8.7}).  There are canonical \(\KK^G\)\nb-equivalences \(qA\to A\) and \(B\to \Comp(L^2G\otimes\ell^2\N)\otimes B\).  We have to compose with the inverses of the maps they induce to extend a functor from \(\Cstarcat^G\) to \(\KKcat^G\).
\end{proof}

Let \(\Ho({\KG(\Unit)}\textrm-\Mod)\) be the homotopy category of right \(\KG(\Unit)\)-modules, obtained by inverting those \(\KG(\Unit)\)-module maps that are stable equivalences in~\(\Sym\), briefly called \emph{weak equivalences}.

\begin{thm}
  \label{the:KG_on_KKG}
  Let~\(G\) be a proper, locally compact, Hausdorff groupoid with Haar system.  
  The functor \(\KG\colon \Cstarcat^G \to {\KG(\Unit)}\textrm-\Mod\) induces a symmetric lax monoidal functor \(\KG\colon \KKcat^G \to \Ho( {\KG(\Unit)}\textrm-\Mod)\) making the following diagram commute:
  \[
  \xymatrix{
    \Cstarcat^G_{\phantom{G}} \ar[d] \ar[r]^-{\KG}
    & {\KG(\Unit)}\textrm-\Mod \ar[d] \\
    \KKcat^G_{\phantom{G}} \ar@{..>}[r]^-{\KG}
    & \Ho( {\KG(\Unit)}\textrm-\Mod).
  }
  \]
  This functor is triangulated.
\end{thm}

\begin{proof}
  For \(A\in \Cstarcat^G\) and \(i\in \Z\), we compute
  \begin{multline}
    \label{eq:Hom_KG_Unit}
    \Ho(\KG(\Unit)\textrm-\Mod)(\KG(\Unit)[i], \KG(A))
    \cong \Ho(\Sym)(\Sphere[i], \KG(A))\\
    \cong \KK^G(\Unit[i], A)
  \end{multline}
  by a standard isomorphism and~\eqref{eq:KG_homotopy}.  A \(\KG(\Unit)\)-module homomorphism is a weak equivalence if it induces an isomorphism on \(\Ho(\Sym)(\Sphere[i], \KG(A))\) for all~\(i\).  Thus \(\KG(f)\) is a weak equivalence if~\(f\) is a \(\KK^G\)\nb-equivalence.  Now Theorem~\ref{thm:factor_through_KK} yields the desired functor \(\KG\colon \KKcat^G\to \Ho( {\KG(\Unit)}\textrm-\Mod)\).  This functor clearly remains symmetric, normal, lax monoidal.  It is also additive.

  The triangulated structure on~\(\KKcat^G\) is defined using \(A\mapsto \Cont_0(\R)\otimes A\) as desuspension and diagrams isomorphic to mapping cone triangles as exact triangles.  We observed in Remark~\ref{rem:extra} that~\(\KG\) maps \(\Cont_0(\R)\otimes A\) to the loop space of~\(\KG(A)\), and the cone of a map~\(f\) to the homotopy fibre of \(\KG(f)\).  Hence the functor \(\KG\colon \KKcat^G\to \Ho( {\KG(\Unit)}\textrm-\Mod)\) is triangulated.
\end{proof}

\section{The Additivity Theorem}
\label{sec:additivity}

In this section we prove the Additivity Theorem.  Let
\(\KKcat^G\) be the triangulated category of
\(G\)\nb-equivariant bivariant \(\KKt\)-theory for a
\emph{cocompact} proper locally compact groupoid~\(G\) with
Haar system.  Recall that this is a symmetric monoidal category
with unit object \(\Unit = \Cont_0(X)\) for the object space~\(X\)
of~\(G\).
Let \(\Tri\defeq \langle \Unit \rangle_\loc\) be the localising triangulated subcategory of \(\KKcat^G\) generated by~\(\Unit\), and let \(\Tri_d\) denote the closed symmetric monoidal category of its dualisable objects, as in Section~\ref{sec:trace}.

\begin{prop}
  \label{prop:dualisable is compact}
  The category~\(\Tri_d\) coincides with~\(\Tri_c\), the full triangulated subcategory of compact\(_{\aleph_1}\) objects, and both are also equal to~\(\langle \Unit \rangle\), the thick triangulated subcategory of~\(\Tri\) \textup(or of \(\KKcat^G\)\textup) generated by the unit object~\(\Unit\).
\end{prop}

\begin{proof}
  Since~\(G\) is cocompact, the \(\Rep(G)\)-modules \(\KKt^G(\Unit[i] , A)\) identify naturally with the topological \(G\)\nb-equivariant \(\K\)\nb-theory \(\Kg_i(A) \cong \K_i(G\ltimes A)\).  Since ordinary \(\K\)\nb-theory of separable \(\Cst\)\nb-algebras yields countable abelian groups and commutes with countable coproducts in \(\KKcat^G\), and since \(G\ltimes \blank\) commutes with coproducts and preserves separability, we conclude that the \(\otimes\)\nb-unit~\(\Unit\) is a compact\(_{\aleph_1}\) object of~\(\KKcat^G\).
  (See \cite{dellAmbrogio:Tensor_triangular}*{\S2.1} for this countable version on the more usual notions of compact objects in triangulated categories; this notion is necessary here because \(\KKcat^G\) only has countable, not arbitrary, coproducts.)

  Thus \(\Tri = \langle \Unit\rangle_\loc\subseteq \KKcat^G\) is compactly\(_{\aleph_1}\) generated, from which it follows that \(\Tri_c=\langle \Unit \rangle\) (\cite{dellAmbrogio:Tensor_triangular}*{Corollary~2.4}).  Moreover, since~\(\Tri\) is generated by the \(\otimes\)\nb-unit, its compact and dualisable objects coincide.  Indeed, it is not difficult to see that~\(\Tri_d\) is a thick triangulated subcategory of~\(\Tri\)
  (this uses that the contravariant Hom functors are sufficiently exact, see~\cite{dellAmbrogio:Tensor_triangular}*{\S2.3}).
  Since the \(\otimes\)\nb-unit is always dualisable, it follows that \(\Tri_c=\langle\Unit\rangle \subseteq \Tri_d\).
  On the other hand, we have seen that $\Unit\in \Tri_c$, and it follows by an easy computation (as in the proof of \cite{Hovey-etal:Axiomatic}*{Theorem 2.1.3\,(a)}) that \(\Tri_d\subseteq \Tri_c\).  Thus \(\Tri_d=\Tri_c\).
\end{proof}

\begin{prop}
  \label{prop:symmetric spectra and trace on KKG}
  The restriction on~\(\langle\Unit\rangle\) of the functor~\(\KG\) of Theorem~\textup{\ref{the:KG_on_KKG}} is a fully faithful, strong monoidal exact functor
  \(\KG\colon \langle\Unit\rangle \to \Ho (\KG(\Unit)\textrm-\Mod)\).
\end{prop}

\begin{proof}
  To show that the restriction of~\(\KG\) on~\(\langle\Unit\rangle\) is a strong monoidal functor, consider in \(\Ho (\KG(\Unit)\textrm-\Mod )\) the natural transformation
  \[
  c\colon \KG(\blank)\wedge_{\KG(\Unit)} \KG(B)\to \KG(\blank\otimes B),
  \]
  where~\(B\) is a fixed separable \(G\)\nb-\(\Cst\)-algebra.  Both the left and the right functors are exact.  By construction, \(c\)~is an isomorphism on suspensions of~\(\Unit\).  It follows by a standard triangulated argument that~\(c\) is an isomorphism on the thick triangulated subcategory of \(\KKcat^G\) generated by~\(\Unit\), which by Proposition~\ref{prop:dualisable is compact} is precisely~\(\langle\Unit\rangle\).

  Thus~\(\KG\) is strong monoidal on~\(\langle\Unit\rangle\).  Similarly, it is fully faithful on~\(\langle\Unit\rangle\) because on the full subcategory of the generators~\(\Unit[i]\) it is an isomorphism
  \begin{multline*}
    \KG\colon
    \KK^G(\Unit[i], \Unit[j])
    \congto
    \Ho(\KG(\Unit)\textrm-\Mod) (\KG(\Unit) [i], \KG(\Unit)[j] ) \\
    \cong \Ho(\Sym)(\Sphere[i], \KG(\Unit)[j])
  \end{multline*}
  by~\eqref{eq:KG_homotopy} (\(i,j\in \Z\)).
\end{proof}

\begin{proof}[Proof of Theorem~\textup{\ref{the:additivity_trace}}]
  We know from Proposition~\ref{prop:dualisable is compact} that the \(G\)\nb-\(\Cst\)-algebras under considerations are dualisable.
  Given a diagram in \(\langle\Unit\rangle\subseteq\KKcat^G\) as in the statement of the theorem,  by Proposition~\ref{prop:symmetric spectra and trace on KKG} we may apply the fully faithful exact functor~\(\KG\) to it to transfer the problem to the tensor triangulated category \(\Ho(\KG(\Unit)\textrm-\Mod)\), where the analogous statement is known to hold.
  Indeed, by \cite{Mandell-etal:Diagram_spectra}*{Theorem 12.1\,(i),(iii)} the module category over any commutative ring spectrum~\(R\) carries a symmetric monoidal model structure (with respect to the smash product~\(\wedge_R\), and with stable equivalences of the underlying symmetric spectra as weak equivalences);
  by \cite{May:Additivity}*{Theorem 1.9 and~\S5--7}), the induced tensor-triangulated structure in the homotopy category of any such model category satisfies additivity of traces as required by Theorem~\ref{the:additivity_trace}, and even in a more general form.

  Since~\(\KG\) is also strong monoidal on \(\langle\Unit\rangle\subseteq\KKcat^G\), by Lemma~\ref{lem:strong_monoidal_duals} it identifies the traces computed in~\(\langle\Unit\rangle\) with those computed in \(\Ho(\KG(\Unit)\textrm-\Mod)\).
\end{proof}

\section{Application to trace computations}
\label{sec:application}

Now we relate traces in triangulated categories to the Hattori--Stallings traces in certain module categories.  We work in the general setting of a tensor triangulated category \((\Tri,\otimes,\Unit)\), assuming additivity of traces.  Let
\[
R\defeq \Tri_*(\Unit,\Unit) = \bigoplus_{n\in\Z} \Tri_n(\Unit,\Unit)
\]
be the graded endomorphism ring of the tensor unit.  It is graded-commutative.

If~\(A\) is any object of~\(\Tri\), then \(M(A)\defeq \Tri_*(\Unit,A)
= \bigoplus_{n\in\Z} \Tri_n(\Unit,A)\) is an \(R\)\nb-module in a
canonical way, and an endomorphism \(f\in\Tri_n(A,A)\) yields a
degree-\(n\) endomorphism \(M(f)\) of \(M(A)\).  We will prove in Theorem~\ref{the:HS_trace} below that, under some
assumptions, the trace of~\(f\) equals the Hattori--Stallings trace of
\(M(f)\) and, in particular, depends only on \(M(f)\).

\begin{example}
  \label{exa:Lefschetz_classical}
  First we consider the example where~\(\Tri\) is the category~\(\KKcat\) of complex separable \(\Cst\)-algebras, \(A=\Cont(X)\) for a compact smooth manifold, and \(f\in\KK_0(A,A)\) is the class of the \(^*\)\nb-homomorphism induced by a smooth self-map of~\(X\) whose graph is transverse to the diagonal.  On the one hand, the trace of~\(f\) in \(\KKcat\) may be computed as a Kasparov product, and the result reduces to the usual expression for the Lefschetz invariant of~\(f\) in terms of fixed points (see~\cite{Emerson-Meyer:Equi_Lefschetz} for this and some equivariant generalisations).  On the other hand, \(\KK_*(\Unit,A)=\K^*(X)\), and the Hattori--Stallings trace of the map induced by~\(f\) agrees with the global cohomological formula for the Lefschetz invariant.  Thus our result generalises the Lefschetz Fixed Point Theorem.
\end{example}

Before we can state our theorem, we must define the Hattori--Stallings trace for endomorphisms of graded modules over graded rings.  This is well-known for ungraded rings (see~\cite{Bass:Euler_discrete}).  The grading causes some notational overhead here.  Let~\(R\) be a (unital) graded-commutative graded ring.  A finitely generated free \(R\)\nb-module is a direct sum of copies of \(R[n]\), where~\(R[n]\) denotes~\(R\) with degree shifted by~\(n\).  Let \(F\colon P\to P\) be a module endomorphism of such a free module, let us assume that~\(F\) is homogeneous of degree~\(d\).  We use an isomorphism
\begin{equation}
  \label{eq:iso_free}
  P\cong\bigoplus_{i=1}^r R[n_i]
\end{equation}
to rewrite~\(F\) as a matrix \((f_{ij})_{1\le i,j\le r}\) with \(R\)\nb-module homomorphisms \(f_{ij}\colon R[n_j]\to R[n_i]\) of degree~\(d\).  The entry~\(f_{ij}\) is given by right multiplication by some element of~\(R\) of degree \(n_i-n_j+d\).  The (super)trace \(\tr F\) is defined as
\[
\tr F \defeq \sum_{i=1}^r (-1)^{n_i} \tr f_{ii},
\]
this is an element of~\(R\) of degree~\(d\).

It is straightforward to check that~\(\tr F\) is well-defined, that is, independent of the choice of the isomorphism in~\eqref{eq:iso_free}.  Here we use that the degree-zero part of~\(R\) is central in~\(R\) (otherwise, we still get a well-defined element in the commutator quotient \(R_d/[R_d,R_0]\)).  Furthermore, if we shift the grading on~\(P\) by~\(n\), then the trace is multiplied by the sign \((-1)^n\) -- it is a supertrace.

If~\(P\) is a finitely generated projective graded \(R\)\nb-module, then \(P\oplus Q\) is finitely generated and free for some~\(Q\), and for an endomorphism~\(F\) of~\(P\) we let
\[
\tr F \defeq \tr (F\oplus 0\colon P\oplus Q\to P\oplus Q).
\]
This does not depend on the choice of~\(Q\).

A \emph{finite projective resolution} of a graded \(R\)\nb-module~\(M\) is a resolution
\begin{equation}
  \label{eq:finite_resolution}
  \dotsb \to P_\ell \xrightarrow{d_\ell}
  P_{\ell-1} \xrightarrow{d_{\ell-1}}
  \dotsb \xrightarrow{d_1} P_0 \xrightarrow{d_0} M
\end{equation}
of finite length by finitely generated projective graded \(R\)\nb-modules~\(P_j\).  \emph{We assume that the maps~\(d_j\) have degree one} (or at least odd degree).  Assume that~\(M\) has such a resolution and let \(f\colon M\to M\) be a module homomorphism.  Lift~\(f\) to a chain map \(f_j\colon P_j\to P_j\), \(j=0,\dotsc,\ell\).  We define the \emph{Hattori--Stallings trace} of~\(f\) as
\[
\tr(f) = \sum_{j=0}^\ell \tr(f_j).
\]
It may be shown that this trace does not depend on the choice of resolution.  It is important for this that we choose~\(d_j\) of degree one.  Since shifting the degree by one alters the sign of the trace of an endomorphism, the sum in the definition of the trace becomes an \emph{alternating} sum when we change conventions to have even-degree boundary maps~\(d_j\).  Still the trace changes sign when we shift the degree of~\(M\).

Recall that \((\Tri,\otimes,\Unit)\) is a tensor triangulated category that satisfies additivity of traces and that \(R\defeq \Tri_*(\Unit,\Unit)\) is a graded-commutative ring.

\begin{thm}
  \label{the:HS_trace}
  Let \(F\in\Tri_d(A,A)\) be an endomorphism of some object~\(A\) of~\(\Tri\).  Assume that~\(A\) belongs to the localising subcategory of~\(\Tri\) generated by~\(\Unit\).  If the graded \(R\)\nb-module \(M(A)\defeq \Tri_*(\Unit,A)\) has a finite projective resolution, then~\(A\) is dualisable in~\(\Tri\) and the trace of~\(F\) is equal to the Hattori--Stallings trace of the induced module endomorphism \(\Tri_*(\Unit,f)\) of~\(M(A)\).
\end{thm}

\begin{proof}
  Our main tool is the phantom tower over~\(A\), which is constructed in~\cite{Meyer:Homology_in_KK_II}.  We recall some details of this construction.

  Let~\(M^\bot\) be the functor from finitely generated projective \(R\)\nb-modules to~\(\Tri\) defined by the adjointness property \(\Tri(M^\bot(P),B) \cong \Tri(P,M(B))\) for all \(B\inOb\Tri\).  The functor~\(M^\bot\) maps the free rank-one module~\(R\) to~\(\Unit\), is additive, and commutes with suspensions; this determines~\(M^\bot\) on objects.  Since \(R=\Tri_*(\Unit,\Unit)\), \(\Tri_*(M^\bot(P_1),M^\bot(P_2))\) is isomorphic (as a graded Abelian group) to the space of \(R\)\nb-module homomorphisms \(P_1\to P_2\).  Furthermore, we have canonical isomorphisms \(M\bigl(M^\bot(P)\bigr) \cong P\) for all finitely generated projective \(R\)\nb-modules~\(P\).

  By assumption, \(M(A)\) has a finite projective resolution as
  in~\eqref{eq:finite_resolution}.  Using~\(M^\bot\), we lift
  it to a chain complex in~\(\Tri\), with entries \(\hat{P}_j
  \defeq M^\bot(P_j)\) and boundary maps \(\hat{d}_j\defeq
  M^\bot(d_j)\) for \(j\ge1\).  The map \(\hat{d}_0\colon
  \hat{P_0}\to A\) is the pre-image of~\(d_0\) under the
  adjointness isomorphism \(\Tri(M^\bot(P),B) \cong
  \Tri(P,M(B))\).  We get back the resolution of modules by
  applying~\(M\) to the chain complex
  \((\hat{P}_j,\hat{d}_j)\).

  Next, it is shown in~\cite{Meyer:Homology_in_KK_II} that we may embed this chain complex into a diagram
  \begin{equation}
    \label{eq:phantom_tower}
    \begin{gathered}
      \xymatrix@C=1em{
        A \ar@{=}[r]&
        N_0 \ar[rr]^{\iota_0^1}&&
        N_1 \ar[rr]^{\iota_1^2} \ar[dl]|-\circ^{\epsilon_0}&&
        N_2 \ar[rr]^{\iota_2^3} \ar[dl]|-\circ^{\epsilon_1}&&
        N_3 \ar[rr] \ar[dl]|-\circ^{\epsilon_2}&&
        \cdots \ar[dl]|-\circ\\
        &&\hat{P}_0 \ar[ul]^{\pi_0}&&
        \hat{P}_1 \ar[ul]^{\pi_1} \ar[ll]|-\circ^{\hat{d}_1}&&
        \hat{P}_2 \ar[ul]^{\pi_2} \ar[ll]|-\circ^{\hat{d}_2}&&
        \hat{P}_3 \ar[ul]^{\pi_3} \ar[ll]|-\circ^{\hat{d}_3}&&
        \cdots \ar[ul] \ar[ll]|-\circ
      }
 \end{gathered}
  \end{equation}
  where the triangles involving~\(\hat{d}_j\) commute and the others
  are exact.  This diagram is called the \emph{phantom tower}
  in~\cite{Meyer:Homology_in_KK_II}.  The circles indicate maps of
  degree one.

  Since \(\hat{P}_j=0\) for \(j>\ell\), the maps~\(\iota_j^{j+1}\) are invertible for \(j>\ell\).  Furthermore, a crucial property of the phantom tower is that these maps~\(\iota_j^{j+1}\) are \emph{phantom maps}, that is, they induce the zero map on \(\Tri_*(\Unit,\blank)\).  Together, these facts imply that \(M(N_j)=0\) for \(j>\ell\).  Since we assumed~\(\Unit\) to be a generator of~\(\Tri\), this further implies \(N_j=0\) for \(j>\ell\).

  Next we recursively extend the endomorphism~\(F\) of \(A=N_0\) to an endomorphism of the phantom tower.  We start with \(F_0=F\colon N_0\to N_0\).  Assume \(F_j\colon N_j\to N_j\) has been constructed.  As in~\cite{Meyer:Homology_in_KK_II}, we may then lift~\(F_j\) to a map \(\hat{F}_j\colon \hat{P}_j\to \hat{P}_j\) such that the square
  \[
  \xymatrix{
    \hat{P}_j \ar[r]^{\pi_j}\ar[d]_{\hat{F}_j}&N_j\ar[d]^{F_j}\\
    \hat{P}_j \ar[r]_{\pi_j}&N_j
  }
  \]
  commutes.  Now we apply the Additivity
  Theorem~\ref{the:additivity_trace} to construct an endomorphism
  \(F_{j+1}\colon N_{j+1}\to N_{j+1}\) such that
  \((\hat{F}_j,F_j,F_{j+1})\) is a triangle morphism and \(\tr (F_j) =
  \tr (\hat{F}_j) + \tr (F_{j+1})\).  Then we repeat the recursion
  step with~\(F_{j+1}\) and thus construct a sequence of
  maps~\(F_j\).  We get
  \[
  \tr(F)
  = \tr(F_0)
  = \tr(\hat{F}_0) + \tr(F_1)
  = \dotsb
  = \tr(\hat{F}_0) + \dotsb + \tr(\hat{F}_\ell) +  \tr(F_{\ell+1}).
  \]
  Since \(N_{\ell+1}=0\), we may leave out the last term.

  Finally, it remains to observe that the trace
  of~\(\hat{F}_j\) as an endomorphism of~\(\hat{P}_j\) agrees
  with the trace of the induced map on the projective
  module~\(P_j\).  Since both traces are additive with respect
  to direct sums of maps, the case of general finitely
  generated projective modules reduces first to free modules
  and then to free modules of rank one.  Both traces change by
  a sign if we suspend or desuspend once, hence we reduce to
  the case of endomorphisms of~\(\Unit\), which is trivial.
  Hence the computation above does indeed yield the
  Hattori--Stallings trace of~\(M(A)\) as asserted.
\end{proof}

\begin{rem}
  Note that if a module has a finite projective resolution, then it
  must be finitely generated.  Conversely, if the graded ring~\(R\) is
  \emph{coherent} and \emph{regular}, then any finitely generated
  module has a finite projective resolution.  (Regular means that every
  finitely generated module has a finite \emph{length} projective resolution;
  coherent means that every finitely generated homogeneous ideal is
  finitely presented~-- for instance, this holds if~\(R\) is (graded)
  Noetherian; coherence implies that any finitely generated graded
  module has a resolution by finitely generated projectives.)

  Moreover, if~\(R\) is coherent, an induction argument shows (as in
  the proof of Proposition~\ref{prop:dualisable is compact}) that for
  every \(A\in \langle \Unit \rangle =(\langle \Unit \rangle_\loc)_d\)
  the module~\(M(A)\) is finitely generated; and if~\(R\) is also
  regular, each such~\(M(A)\) has a finite projective resolution.

  In conclusion: if~\(R\) is regular and coherent, an object \(A\in
  \langle \Unit \rangle_\loc\) is dualisable if and only if~\(M(A)\)
  has a finite projective resolution.
\end{rem}

\begin{proof}[Proof of Theorem~\textup{\ref{thm:HS_for_KKG}}.]
  We specialise to equivariant Kasparov theory for a compact
  group~\(G\) and complex \(\Cst\)\nb-algebras.  In this case,
  the tensor unit is \(\Unit=\C\) and the ring
  \(R=\KK^G_*(\C,\C)\) becomes the tensor product of the
  representation ring \(\Rep(G)\) of~\(G\) with the ring of
  Laurent polynomials \(\Z[\beta,\beta^{-1}]\) in one
  variable~\(\beta\) of degree~\(2\), which generates Bott
  periodicity.  As a result, the category of graded
  \(R\)\nb-modules is equivalent to the category of
  \(\Z/2\)-graded \(\Rep(G)\)-modules.  Such an object is
  projective or finitely generated if and only if its even and
  odd part are projective or finitely generated
  \(\Rep(G)\)-modules, respectively.  Hence we are essentially
  dealing with pairs of ungraded modules over the ungraded ring
  \(\Rep(G)\).

  The Hattori--Stallings trace of an even-degree endomorphism
  of a \(\Z/2\)-graded \(\Rep(G)\)-module \((M_+,M_-)\) is the
  \emph{difference} of the Hattori--Stallings traces on the
  \(\Rep(G)\)-modules \(M_+\) and~\(M_-\), respectively, where
  we are now talking about ungraded modules over ungraded
  rings.

  The modules we need are of the form \(\Tri_*(\Unit,A)\) for a
  separable \(G\)\nb-\(\Cst\)-algebra~\(A\), and this is
  isomorphic to the \(G\)\nb-equivariant \(\K\)\nb-theory
  \(\K_*^G(A) \cong \KK_*^G(\C,A)\).
  Theorem~\ref{the:HS_trace} only applies if~\(A\) belongs to
  the localising subcategory generated by~\(\Unit\).  For a
  finite group~\(G\), this is a rather unnatural assumption.
  Furthermore, in that case there are many modules without
  finite length projective resolutions.

  Under the assumption that~\(G\) is a connected Lie group with
  torsion-free fundamental group (also called \emph{Hodgkin Lie
    group}), results of~\cite{Meyer-Nest:BC_Coactions} imply
  that the localising subcategory~\(\langle \Unit
  \rangle_\loc\) of \(\KKcat^G\) contains a separable
  \(G\)\nb-\(\Cst\)-algebra~\(A\) if and only if~\(A\) belongs
  to the non-equivariant bootstrap class in \(\KKcat\), that
  is, there is an invertible element in \(\KK(A,B)\) for a
  commutative \(\Cst\)\nb-algebra~\(B\) (no
  \(G\)\nb-equivariance is required).  Furthermore, the ring
  \(\Rep(G)\) is regular and Noetherean.  Thus a
  (\(\Z/2\)-graded) \(\Rep(G)\)-module has a finite projective
  resolution if and only if it is finitely generated as a
  \(\Rep(G)\)-module.  Thus Theorem~\ref{thm:HS_for_KKG} is a
  special case of Theorem~\ref{the:HS_trace}.
\end{proof}

\end{document}